\documentclass[11pt,reqno]{amsart}
\usepackage{amssymb}
\usepackage{amsfonts}
\usepackage{amssymb}

\allowdisplaybreaks
\numberwithin{equation}{section}

\begin{document}
    \theoremstyle{plain}
    \newtheorem{thm}{Theorem}[section]
    \newtheorem{theorem}[thm]{Theorem}
    \newtheorem{lemma}[thm]{Lemma}

    \newtheorem{corollary}[thm]{Corollary}
    \newtheorem{corollary*}[thm]{Corollary*}
    \newtheorem{proposition}[thm]{Proposition}
    \newtheorem{proposition*}[thm]{Proposition*}
    \newtheorem{conjecture}[thm]{Conjecture}
    \theoremstyle{definition}
    \newtheorem{construction}[thm]{Construction}
    \newtheorem{notations}[thm]{Notations}
    \newtheorem{question}[thm]{Question}
    \newtheorem{problem}[thm]{Problem}
    \newtheorem{remark}[thm]{Remark}
    \newtheorem{remarks}[thm]{Remarks}
    \newtheorem{definition}[thm]{Definition}
    \newtheorem{claim}[thm]{Claim}
    \newtheorem{assumption}[thm]{Assumption}
    \newtheorem{assumptions}[thm]{Assumptions}
    \newtheorem{properties}[thm]{Properties}
    \newtheorem{example}[thm]{Example}
    \newtheorem{comments}[thm]{Comments}
    \newtheorem{blank}[thm]{}
    \newtheorem{observation}[thm]{Observation}
    \newtheorem{defn-thm}[thm]{Definition-Theorem}


    \title{Affine geometry and Frobenius algebra}
    \author{Kefeng Liu}
    \address{Department of Mathematics,University of California at Los
    Angeles, Box 951555, Los Angeles, CA, 90095-1555, USA.}
    \email{liu@math.ucla.edu}
    \author{Hao Xu}
    \address{Center of Mathematical Sciences, Zhejiang University, Hangzhou, Zhejiang 310027, China}
    \email{mathxuhao@gmail.com}
    \author{Yanhui Zhi}
    \address{Center of Mathematical Sciences, Zhejiang University, Hangzhou, Zhejiang 310027, China}
    \email{21935038@zju.edu.cn}

\begin{abstract}
The associativity of the multiplication on a Frobenius manifold is
equivalent to the WDVV equation of a symmetric cubic form in flat
coordinates. Frobenius manifold could be regarded a very special
type of statistical manifold. There is a natural commutative product
on each tangent space of a statistical manifold. We show that it is
associative, hence making it into a manifold with Frobenius algebra
structure, if and only if the sectional $K$-curvature vanishes. In
other words, WDVV equation is equivalent to zero sectional
$K$-curvature. This gives a curvature interpretation for WDVV
equation.
\end{abstract}

    \maketitle

    \section{Introduction}

The classical affine differential geometry studies properties of
hypersurfaces in $\mathbb R^{n+1}$ that are invariant under affine
transformations. The work of Calabi and Cheng-Yau on affine
differential geometry played an important role in Yau's proof of the
Calabi conjecture.

What we are going to study in this paper is usually called
information geometry or statistical geometry, which belongs to
affine differential geometry in the broad sense.

A statistical manifold is defined to be a Riemannian manifold
equipped with two torsion-free connections dual to each other with
respect to the Riemannian metric, or equivalently defined as a
Riemannian manifold with a totally symmetric $(1,2)$-tensor. It was
so named because some early examples of stastistical manifolds are
from families of probability distributions.

Locally strictly convex equiaffine hypersurfaces in $\mathbb R^n$
and Hessian manifolds are examples of statistical manifolds. Many
results in classical affine geometry could be studied from the point
of view of statistical manifolds. If the two dual connections are
both flat, we arrived at the dually flat structure introduced by
Amari-Nagaoka \cite{Ama} with wide applications in modern
information geometry, statistics and related fields.

Hessian metric was studied by Cheng-Yau \cite{CY} as an analogue of
a K\"ahler metric for flat affine structures. Cheng-Yau also proved
existence and uniqueness theorems for Hessian metric, which played
important role in mirror symmetry.

Frobenius manifold, or more generally manifold with Frobenius
algebra structure on their tangent bundle, is also a statistical
manifold, as they both have a symmetric $(0,3)$-tensor. But this
connection seems was largely overlooked, partly because the
definition of Frobenius manifold requires several very restrictive
properties and mostly being studied in the complex analytic setting.

Dubrovin \cite{Dub} introduced Frobenius manifolds as a geometric
framework for WDVV equation in 2D topological field theory, thus
unified Saito's unfolding spaces of singularities \cite{Sai} and
quantum cohomology \cite{Man}. Another prominent example of
Frobenius manifold is Barannikov-Kontsevich's construction \cite{BK}
from Batalin-Vilkovisky algebras in mirror symmetry.

Dubrovin \cite[p.311]{Dub} remarked the similarity of the first
structure connection on Frobenius manifold and the
$\alpha$-connection on statistical manifold, both of which are
linear deformations of Levi-Civita connection.

A statistical manifold has natural commutative product on its
tangent space \cite{CCN,CM,JTZ,NO}. Combe-Manin \cite{CM} showed
that various spaces of probability distributions carry natural
structures of $F$-manifolds, a weakened version of Frobenius
manifold. Jiang-Tavakoli-Zhao \cite{JTZ} and Nakajima-Ohmoto
\cite{NO} showed that dually flat structure (with flat metric)
implies Frobenius algebraic structure. Combe-Combe-Nencka \cite{CCN}
proved that statistical manifolds, related to exponential families
and with flat structure connection have a Frobenius manifold
structure.

The vanishing of WDVV equation characterized those statistical
manifolds with Frobenius structure on the tangent space. In this
paper, we give another characterization in terms of the sectional
$K$-curvature, a concept from affine differential geometry due to
Opozda \cite{Opo}. More precisely, we show that the associativity
holds if and only if the sectional $K$-curvature is zero. In fact,
the latter condition means the commutativity of left multiplication
operators. We provide alternative proofs of some results of Opozda
\cite{Opo} with the help of Frobenius structure. The 2-dimensional
example at the end of the paper shows that at least locally, given
any Riemannian metric, there is abundance of statistical structure
whose tangent space has Frobenius structure.

    \vskip 30pt
    \section{Statistical structure and sectional $K$-curvature}\label{sec:stat}

    Throughout the paper, a connection always means an affine
    connection on the tangent bundle. We follow the notation of
    \cite{Opo}.
    Let $\nabla$ be a connection on a Riemannian manifold $(M,g)$. Then
        the dual connection $\overline{\nabla}$ is defined to be the unique
        connection that satisfies
        $$X(g(Y,Z))=g(\nabla_X Y,Z)+g(Y,\overline\nabla_X Z)$$
        for all $X,Y,Z\in\mathfrak{X}(M)$.

If $\nabla$ is a torsion-free connection on $(M,g)$, then
$\overline\nabla$ is torsion-free if and only if $\nabla g$ is
symmetric as a $(0,3)$-tensor. A triple $(M,g,\nabla)$ is called a
statistical
    manifold if both $\nabla$ and $\overline\nabla$ are torsion-free.

On a statistical manifold, the Levi-Civita connection $\hat\nabla$
satisfies $\hat\nabla=\frac12(\nabla+\overline\nabla)$. The
Amari-Chentsov tensor $T$ is the difference of Christoffel symbols
of $\nabla$ and $\overline\nabla$, namely
\begin{equation}\label{eqC}
T_{ijk}=\overline\Gamma_{ijk}-\Gamma_{ijk}.
\end{equation}
Here
$\Gamma_{ijk}=\Gamma_{ij}^h g_{hk}$ and similarly for
$\overline\Gamma_{ijk}$.

 \begin{lemma}\cite[Theorem 6.1]{Ama}
On a statistical manifold $(M,g,\nabla)$, the Amari-Chentsov tensor
satisfies
$$T_{ijk}=\nabla_i g_{jk}.$$
In particular, the Amari-Chentsov tensor is symmetric.
 \end{lemma}
\begin{proof}
We have
$$\nabla_i g_{jk}=\partial_i g_{jk}-\Gamma_{ikj}-\Gamma_{ijk}.$$
and
$$0=\hat\nabla_i g_{jk}=\partial_i g_{jk}-\hat\Gamma_{ikj}-\hat\Gamma_{ijk}.$$
Since $\hat\Gamma_{ijk}-\Gamma_{ijk}=\frac{1}{2}T_{ijk}$, we have
\begin{equation}\label{eqC2}
\nabla_i g_{jk}=\frac12(T_{ikj}+T_{ijk}).
\end{equation}
Since $\nabla_i g_{jk}$ is symmetric for $i$ and $k$, the above
equation implies
$$T_{ikj}+T_{ijk}=T_{kij}+T_{kji}.$$
By definition $T_{ijk}$ is symmetric for $i$ and $j$, the above
equation implies
$$T_{ijk}=T_{kji}.$$
Namely $T_{ijk}$ is also symmetric for $i,k$. Hence $T$ is symmetric
and we get $\nabla_i g_{jk}=T_{ijk}$ from \eqref{eqC2}.
\end{proof}

Another commonly used tensor is $K_XY$ (also denoted by $K(X,Y)$)
$$K_XY:=\nabla_X Y-\hat\nabla_X Y$$
Obviously $T(X,Y,Z)=-2g(K(X,Y),Z)$.

A statistical manifold can be equivalently defined as a triple
$(M,g,K)$ where $K$ is a $(1,2)$-tensor such that $g(K(X,Y),Z)$ is
symmetric in $X,Y,Z$. This is because if such $K$ is given, then
$\nabla=\hat\nabla+K$ is a torsion-free connection and
$$\nabla_X g(Y,Z)=\hat\nabla_Xg(Y,Z)-g(K_X Y,Z)-g(Y,K_X
Z)=-2g(K(X,Y),Z).$$ In particular, $\nabla g$ is symmetric, hence
$\overline\nabla$ is torsion-free.

The Riemannian curvature tensors of $\nabla$ and $\overline\nabla$
satisfy \cite[Proposition 4.6]{NS}
\begin{equation}\label{eqR}
g(R(Z,W)X,Y)+g(\overline R(Z,W)Y,X)=0.
\end{equation}
In particular, $\nabla$ is flat if and only if $\overline\nabla$ is
flat. From \eqref{eqR}, we get the following criterion for
$R=\overline R$.
\begin{corollary}
    For a statistical manifold,
$R =\overline R$ if and only if $g(R(X,Y)Z,W)$ is skew-symmetric for
$Z,W$.
    \end{corollary}

We also have the following relations of $R,\overline R$ and $\hat
R$.
\begin{lemma}\cite{Opo}
On a statistical manifold, the Riemannian curvature tensors of
$\nabla$, $\overline\nabla$ and $\hat\nabla$ satisfy
\begin{align} \label{eqR2}
R(X,Y) + \overline R(X,Y) &= 2\hat R(X,Y) + 2[K_X, K_Y],\\
\label{eqR3} R(X,Y) - \overline R(X,Y) &= 2(\hat\nabla_X K)_Y -
2(\hat\nabla_Y K)_X
\end{align}
\end{lemma}
\begin{proof}
Since $\nabla_X Y=\hat\nabla_X Y+K_X Y$ and $\overline\nabla_X
Y=\hat\nabla_X Y-K_X Y$, it is sufficient to prove
\begin{equation}\label{eqR4}
R(X,Y) = \hat R(X,Y) + (\hat\nabla_X K)_Y - (\hat\nabla_Y K)_X +
[K_X, K_Y].
\end{equation}
We have
\begin{align*}
R(X,Y)Z&=\nabla_X\nabla_Y Z-\nabla_Y\nabla_X Z-\nabla_{[X,Y]} Z\\
&=\nabla_X(\hat\nabla_Y Z+K_Y Z)-\nabla_Y(\hat\nabla_X Z+K_X
Z)-\hat\nabla_{[X,Y]}Z-K_{[X,Y]}Z\\
&=\hat\nabla_X\hat\nabla_Y Z+K_X\hat\nabla_Y Z+\hat\nabla_X K_Y
Z+K_X K_Y Z\\
 &\qquad -\hat\nabla_Y\hat\nabla_X Z-K_Y\hat\nabla_X
Z-\hat\nabla_Y K_X Z-K_Y K_X Z-\hat\nabla_{[X,Y]}Z-K_{[X,Y]} Z\\
&=(\hat\nabla_X\hat\nabla_Y Z-\hat\nabla_Y\hat\nabla_X
Z-\hat\nabla_{[X,Y]}Z)+(\hat\nabla_X K_Y Z-K_Y\hat\nabla_X
Z-K_{\hat\nabla_X Y} Z)\\
&\qquad -(\hat\nabla_Y K_X Z-K_X\hat\nabla_Y Z-K_{\hat\nabla_Y X}
Z)+(K_X K_Y Z-K_Y K_X Z)\\
&=\hat R(X,Y)Z + (\hat\nabla_X K)_Y Z - (\hat\nabla_Y K)_X Z+ [K_X,
K_Y]Z.
\end{align*}
In the next to last equation, we used the torsion-freeness
$[X,Y]=\hat\nabla_X Y-\hat\nabla_Y X$.
\end{proof}

On a statistical manifold $(M,g,K)$, we have a family of
torsion-free connections
\begin{equation}\label{eqAlpha}
\nabla^{(\alpha)}_X Y=\hat\nabla_X Y+\alpha K_X Y,\qquad
\alpha\in\mathbb R,
\end{equation}
called $\alpha$-connections. In the theory of Frobenius manifolds,
these are called Dubrovin's first structure connections \cite{Man}
and are assumed to be flat for all $\alpha$.

\begin{lemma}\label{Lem:Alpha} Assume that there exist $\beta,\gamma\in\mathbb R$ with $\beta+\gamma\neq 0$ such that
$\nabla^{(\beta)}$ and $\nabla^{(\gamma)}$ are flat. Then
$\nabla^{(\alpha)}$ are flat for all $\alpha\in\mathbb R$.
\end{lemma}
\begin{proof}
Recall Zhang's formula \cite{Zha} of Riemannian curvature tensor of
$\nabla^{(\alpha)}$ ,
\begin{equation}\label{eqR5}
R^{(\alpha)}(X,Y)=\frac{1+\alpha}{2}R(X,Y) +
\frac{1-\alpha}{2}\overline R(X,Y) -(1-\alpha^2) [K_X, K_Y],
\end{equation}
where $R(X,Y)$ and $\overline R(X,Y)$ are Riemannian curvature
tensors of $\nabla=\nabla^{(1)}$ and $\overline\nabla=\nabla^{(-1)}$
respectively. When $\alpha=0$, \eqref{eqR5} is just \eqref{eqR2}.

From \eqref{eqR5}, we get for all $\alpha\in\mathbb R$,
\begin{equation}\label{eqR6}
R^{(\alpha)}(X,Y)-R^{(-\alpha)}(X,Y)=\alpha(R(X,Y)-\overline
R(X,Y)).
\end{equation}

Since $\nabla^{(-\alpha)}$ is dual to $\nabla^{(\alpha)}$, we have
that $\nabla^{(-\beta)}$ and $\nabla^{(-\gamma)}$ are also flat.
Taking $\alpha=\beta$ and $\alpha=\gamma$ in \eqref{eqR6}, we get
$R(X,Y)=\overline R(X,Y)$. Hence \eqref{eqR5} becomes
\begin{equation}\label{eqR7}
R^{(\alpha)}(X,Y)=R(X,Y) -(1-\alpha^2) [K_X, K_Y].
\end{equation}
Taking $\alpha=\beta$ and $\alpha=\gamma$ in \eqref{eqR7} and noting
that $1-\beta^2\neq 1-\gamma^2$, we get $[K_X, K_Y]=0$. Then
\eqref{eqR7} implies that $R^{(\alpha)}(X,Y)=R(X,Y)$ for all
$\alpha\in\mathbb R$. So $R^{(\alpha)}(X,Y)=0$ for all
$\alpha\in\mathbb R$.
\end{proof}

On a statistical manifold, Opozda \cite{Opo} introduced the concept
of sectional $K$-curvature. First, Opozda noted that the
$(1,3)$-tensor $[K,K]$ given by
    \begin{equation}
        [K,K](X,Y)Z := [K_X,K_Y]Z = K_X K_Y Z - K_Y K_X Z
    \end{equation}
satisfies the same properties as the Riemannian curvature tensor,
except the second Bianchi identity.
\begin{lemma}\label{Lem:KK}\cite{Opo} The following equations of $[K,K]$ hold.
    \begin{itemize}
    \item[(i)]
        $[K,K](X,Y) = - [K,K](Y,X)$.
    \item[(ii)]
        $[K,K](X,Y)Z + [K,K](Y,Z)X + [K,K](Z,X)Y = 0$.
    \item[(iii)]
        $g([K,K](X,Y)Z,W) = -g([K,K](X,Y)W,Z)$.
    \item[(iv)]
        $g([K,K](X,Y)Z,W)=g([K,K](W,Z)Y,X)$.
    \end{itemize}
\end{lemma}
\begin{proof}
(i) follows from the definition. (ii) follows from \eqref{eqR2},
since the first Bianchi identity holds for all of $R,\overline R$
and $\hat R$.

For (iii), we first apply \eqref{eqR2} and then \eqref{eqR} to get
\begin{align*}
&g([K,K](X,Y)Z,W)+g([K,K](X,Y)W,Z)\\
=&g(R(X,Y)Z,W)+g(\overline R(X,Y)Z,W)-2g(\hat R(X,Y)Z,W)\\
&\qquad +g(R(X,Y)W,Z)+g(\overline R(X,Y)W,Z)-2g(\hat R(X,Y)W,Z)\\
=&-g(\overline R(X,Y)W,Z)+g(\overline R(X,Y)Z,W)-g(\overline
R(X,Y)Z,W)+g(\overline R(X,Y)W,Z)\\
=&0.
\end{align*}

For (iv), we use the symmetry of $K$ to get
\begin{align*}
g([K,K](X,Y)Z,W)&=g(K_X K_Y Z,W)-g(K_Y K_X Z,W)\\
&=g(K(X,W),K(Y,Z))-g(K(Y,W),K(X,Z)).
\end{align*}
Apply the above calculation to $g([K,K](W,Z)Y,X)$, we get the same
expression.
\end{proof}

It follows that for any plane $\pi$ in the tangent space $T_pM$, we
can define its sectional $K$-curvature
\begin{equation}
k(\pi)=g([K,K](e_1,e_2)e_2,e_1),
\end{equation}
where $e_1,e_2$ is an orthonormal basis of $\pi$. Note that $k(\pi)$
is independent of the choice of the orthonormal basis.

\begin{lemma}\label{Lem:Const}\cite{Opo}
The following statements are equivalent (at a given point $p\in M$).
\begin{itemize}
\item[(i)] The sectional $k$-curvature is equal to a constant $A$ for all
planes.
\item[(ii)] $[K,K](X,Y)Z=A[g(Y,Z)X-g(X,Z)Y]$.
\item[(iii)] $g(K(X,W),K(Y,Z))-g(K(Y,W),K(X,Z))=A[g(X,W)g(Y,Z)-g(Y,W)g(X,Z)]$.
\end{itemize}
\end{lemma}

\begin{lemma}\label{Lem:ZeroK}
For a statistical manifold $(M,g,K)$, the following statements are
equivalent.
\begin{itemize}
\item[(i)]
The sectional $K$-curvatures are constantly zero.

\item[(ii)]
$[K,K]=0$.

\item[(iii)]
$R(X,Y) + \overline R(X,Y) = 2\hat R(X,Y)$.
\end{itemize}
In fact, the equivalence holds pointwise.
\end{lemma}
\begin{proof}
The equivalence of (ii) and (iii) follows from \eqref{eqR2}. The
equivalence of (i) and (ii) amounts to that the sectional
$K$-curvatures of all planes at a point determines the tensor
$[K,K]$. It is well-known that the Riemannian curvature tensor is
determined by sectional curvatures.

Write $S(X,Y,Z,W)=g([K,K](X,Y)W,Z)$. In fact, one can prove that if
$S$ is a $(0,4)$-tensor with properties of Lemma \ref{Lem:KK} and
$S(X,Y,X,Y)=0$ for all $X,Y\in T_p M$, then $S=0$ at $p$. The
argument is the same as that for the Riemannian curvature tensor.
\end{proof}

The following lemma is the key tool used by Opozda \cite{Opo} to
study statistical manifolds with constant sectional $K$-curvature.
\begin{lemma}\label{Lem:Basis}\cite{Opo} At a given point $p$ of a statistical
manifold $(M,g,K)$, if the sectional $K$-curvature is equal to
constant $A$, then there is an orthonormal basis $e_1,...,e_n$ of
$T_p M$ and numbers $\lambda _i$, $\mu _i$ for $1\leq i\leq n$ such
that
\begin{equation}\label{eqK}
K(e_i,e_i)=\mu _1e_1 +...+\mu _{i-1}e_{i-1}+\lambda _i e_i
\end{equation}
for $1\leq i\leq n$ and
\begin{equation}\label{eqK2}
K(e_i,e_j)= \mu _{i}e_j
\end{equation}
for $1\leq i<j\leq n$. Moreover $\mu_i$ are determined by
$\lambda_k$ and $A$ through the formula
\begin{equation}\label{eqmu}
\mu _i= \frac{\lambda_i -\sqrt{\lambda _i^2
-4A_{i-1}}}{2},\end{equation}
\begin{equation}\label{eqA}
A_i=A_{i-1}-\mu _{i}^2
\end{equation}
for $1\leq i\leq n$ and $A_0=A$. In \eqref{eqmu}, we require
$\lambda _i^2 -4A_{i-1}\geq0$.
\end{lemma}

\begin{remark}\label{Rm:Algo}
The proof is by induction. Opozda's method to obtain the orthonormal
basis $e_1,...,e_n$ in the above lemma is as follows. Let
\begin{equation}
C(X,Y,Z)=g(K(X,Y),Z).
\end{equation}
Then $C$ is a symmetric $(0,3)$-tensor. Denote by $S^1$ the unit
sphere in $T_p M$ and by $\Phi$ the function $\Phi(X)=C(X,X,X)$ on
$S^1$. The vector $e_1\in S^1$ is any unit vector at which $\Phi$
attains a local maximum, and $e_2\in \{e_1\}^\perp\cap S^1$ is any
unit vector at which $\Phi_{|\{e_1\}^\perp\cap S^1}$ attains a local
maximum, etc.
\end{remark}

\begin{corollary}\label{Cor:Zero}\cite{Opo}
If $A=0$ in Lemma \ref{Lem:Basis}, then there is an orthonormal
basis $e_1,...,e_n$ of $T_p M$ and numbers $\lambda _i$, $\mu _i$
for $1\leq i\leq n$ such that
\begin{equation}\label{eqK7}
K(e_i,e_i)=\lambda_i e_i,\qquad K(e_i,e_j)=0
\end{equation}
for $i\neq j$.
\end{corollary}

Corollary \ref{Cor:Zero} follows immediately from Lemma
\ref{Lem:Basis}. A non-inductive proof is as follows. Since $g(K_X
Y,Z)=g(Y,K_X Z)$, we see that $K_X$ is self-adjoint and hence
diagonalizable. Since $[K,K]=0$, these $K_X$ commute with each
other, hence they are simultaneously diagonalizable with respect to
an orthonormal basis $e_1,\dots,e_n$.

\begin{definition}
A $(1,2)$-tensor $K$ on a Riemannian manifold $(M,g)$ is called
totally symmetric if the corresponding $(0,3)$-tensor
$$C(X,Y,Z)=g(K(X,Y),Z)$$
is symmetric.
\end{definition}

The following lemma is expected.
\begin{lemma}
If $K$ has expression as in \eqref{eqK} and \eqref{eqK2} with
numbers $\lambda_i,\mu_i$ satisfying \eqref{eqmu} and \eqref{eqA},
then $K$ is totally symmetric and has constant sectional
$K$-curvature equal to $A$.
\end{lemma}
\begin{proof}
To show $K$ is totally symmetric, we only need to check
\begin{equation}\label{eqK6}
g(K(e_i,e_j),e_k)=g(K(e_i,e_k),e_j).
\end{equation}
Without loss of
generality, we may assume $i\leq j$.

If $i,j,k$ are distinct, then $i<j$,
$$g(K(e_i,e_j),e_k)=g(\mu_i e_j,e_k)=0,$$
and
\begin{equation*}
g(K(e_i,e_k),e_j) = \begin{cases} g(\mu_i e_k,e_j)=0 &\mbox{if }
i<k,
\\ g(\mu_k e_i,e_j)=0 & \mbox{if } i>k.
\end{cases}
\end{equation*}

If exactly two of $i,j,k$ are equal, we may take $i=j$, then
\begin{equation*}
g(K(e_i,e_i),e_k) =g(\mu_1 e_1+\cdots+\mu_{i-1}e_{i-1}+\lambda_i
e_i, e_k)= \begin{cases} 0 &\mbox{if } i<k,
\\ \mu_k & \mbox{if }
i>k.
\end{cases}
\end{equation*}
and
\begin{equation*}
g(K(e_i,e_k),e_i) = \begin{cases} g(\mu_i e_k,e_i)=0 &\mbox{if }
i<k,
\\ g(\mu_k e_i,e_i)=\mu_k & \mbox{if } i>k.
\end{cases}
\end{equation*}
In both cases, we have proved \eqref{eqK6}.

Now we show that $K$ has constant sectional $K$-curvature equal to
$A$. By Lemma \ref{Lem:Const}, it is sufficient to prove that for
all $i,j,k,l$,
\begin{equation}\label{eqK3}
g(K(e_i,e_k),K(e_j,e_l))-g(K(e_j,e_k),K(e_i,e_l))=A(\delta_{ik}\delta_{jl}-\delta_{jk}\delta_{il}).
\end{equation}
The left-hand side of \eqref{eqK3} is just
$g([K,K](e_i,e_j)e_l,e_k)$.

By the symmetries of $[K,K]$ in Lemma \ref{Lem:KK}, we could assume
$i<j,k<l,i\leq k$. Denote by $LH$ the left-hand side and $RH$ the
righ-hand side of \eqref{eqK3} respectively. We treat six cases
separately.

Case 1. $i<j<k<l$.

Case 2. $i<j=k<l$.

Case 3. $i<k<j<l$.

In all of the above three cases, obviously $LH=RH=0$.

Case 4. $i<k<j=l$. Then $RH=0$ and
$$LH=g(\mu_i e_k,\mu_k e_k)-g(\mu_k e_l,\mu_i
e_l)=0.$$

Case 5. $i=k<l=j$. Then $RH=A$. Using
$\lambda_i\mu_i-\mu_i^2=A_{i-1}$ and $\mu_i^2+A_i=A_{i-1}$,
\begin{align*}
LH&=g(K(e_k,e_k),K(e_j,e_j))-g(K(e_j,e_k),K(e_k,e_j))\\
&=(\mu_1^2+\cdots+\mu_{k-1}^2+\lambda_k\mu_k)-\mu_k^2\\
&=\mu_1^2+\cdots+\mu_{k-1}^2+A_{k-1}\\
&=A.
\end{align*}

Case 6. $i\leq k<l<j$. Then obviously $LH=RH=0$.
\end{proof}


Consider a family of probability distributions $$M=\{p(x,\theta)\mid
\theta\in\Theta\},\quad \text{$\Theta$ is a domain in $\mathbb
R^n$}$$ For each $\theta$, $p(x,\theta)$ is defined on a measure
space $ \mathcal X$ and satisfies $\int_{\mathcal
X}p(x,\theta)dx=1$.

Let $\ell_\theta(x)=\log p(x,\theta)$. Under mild conditions, $M$ is
a manifold with a Riemannian metric
$$g_{ij}(\theta)=E_\theta\left(\frac{\partial \ell_\theta(x)}{\partial \theta^i}\frac{\partial \ell_\theta(x)}{\partial \theta^j}\right)=
\int_{\mathcal X}\frac{\partial \ell_\theta(x)}{\partial
\theta^i}\frac{\partial \ell_\theta(x)}{\partial
\theta^j}p(x,\theta)dx,$$ called the Fisher information matrix. A
statistical structure on $M$ is given by taking
\begin{equation}
C_{ijk}=-\frac{1}{2}E_\theta\left(\frac{\partial
\ell_\theta(x)}{\partial \theta^i}\frac{\partial
\ell_\theta(x)}{\partial \theta^j}\frac{\partial
\ell_\theta(x)}{\partial \theta^k}\right),
\end{equation}
or equivalently in terms of Amari-Chentsov tensor
\begin{equation}\label{eqAC}
T_{ijk}=E_\theta\left(\frac{\partial \ell_\theta(x)}{\partial
\theta^i}\frac{\partial \ell_\theta(x)}{\partial
\theta^j}\frac{\partial \ell_\theta(x)}{\partial \theta^k}\right).
\end{equation}

\begin{example}\label{Ex:exp}
The exponential family consists of probability distributions of the
form
$$p(x,\theta)=\exp\{Q(x)+\sum_{i=1}^n F_i(x)\theta^i-\varphi(\theta)\}$$

We know that the Fisher information matrix of exponential family is
given by
$$g_{ij}(\theta)=\frac{\partial \varphi}{\partial
\theta^i\partial\theta^j},$$
which is a Hessian metric.

Denote by $T$ the Amari-Chentsov tensor \eqref{eqAC}. Then
$T_{ijk}(\theta) =
\partial_k(g_{ij}) = \partial_i \partial_j \partial_k
\varphi(\theta)$. The Christoffel symbols and curvature tensors of
the $\alpha$-connections are given by
\begin{align*}
\Gamma^{(\alpha)}_{ijk}&= \frac{1-\alpha}{2} T_{ijk}(\theta),\\
R^{(\alpha)}_{ijkl}&= \frac{1-\alpha^2}{4}g^{pq}(T_{ilp}T_{jkq} -
T_{ikp}T_{jlq}).
\end{align*}
So both $\nabla=\nabla^{(1)}$ and $\overline\nabla=\nabla^{(-1)}$
are flat. Namely the exponential family is dually flat.

Write $S(X,Y,Z,W)=g([K,K](X,Y)W,Z)$. Then by \eqref{eqR2},
\begin{equation}
S_{ijkl}=-\hat R_{ijkl}
\end{equation}
Namely the sectional $K$-curvature $k(\pi)$ and the sectional
curvature $\hat k(\pi)$ for exponential family are related by
$k(\pi)=-\hat k(\pi)$ for any plane $\pi$.

If we take $(M,g,\nabla^{(\alpha)})$ as the statistical structure,
namely the totally symmetric $(1,2)$-tensor on $M$ is now $\alpha
K$. Then $k(\pi)=-\alpha^2\hat k(\pi)$ for any plane $\pi$. In fact,
this relation holds for any Hessian manifold.
\end{example}

 \vskip 30pt
    \section{Frobenius structure}\label{sec:Frob}

    \begin{definition} A Frobenius algebra $V$ is a
    finite-dimensional commutative associative algebra (with unit)
    over $\mathbb R$ (or $\mathbb C$) that satisfies either of the
    following two equivalent conditions:
    \begin{itemize}
    \item[(i)]
    There is a non-degenerate inner product $g$ such that
     $$g(ab,c)=g(a,bc).$$

    \item[(ii)]
   There is a linear form $\theta: V\rightarrow\mathbb R$ such that
   $$g(a,b)=\theta(ab)$$
   is a non-degenerate inner product.
    \end{itemize}
    \end{definition}

\begin{remark}\label{Rm:SS}
When $g$ is positive-definite, the Frobenius algebra is semisimple,
which means that it is isomorphic to $\mathbb R^n$ (or $\mathbb
C^n$) with component-wise multiplication.
\end{remark}


A Frobenius manifold is a manifold with a smoothly varying Frobenius
algebra structure on the tangent space. The non-degenerate inner
product $g$ serves as a pseudo-Riemannian metric. On such manifold,
we have a symmetric $(0,3)$-tensor
\begin{equation}\label{eqC3}
C(X,Y,Z)=g(XY,Z)=\theta(XYZ).
\end{equation}
So it is a statistical manifold (with pseudo-Riemannian metric).

The full definition of a Frobenius manifold requires some additional
conditions like $g$ is flat, $\hat\nabla C$ is symmetric and
$\hat\nabla e$, where $e$ is unit vector field. Dubrovin's first
structure connection on a (complex) Frobenius manifold is given by
$$\nabla_XY=\hat\nabla_XY+\lambda XY,$$
where $\lambda\in\mathbb C$. Comparing it with the definition of
$\alpha$-connection \eqref{eqAlpha} naturally leads to the following
definition of a commutative product on the tangent space of a
statistical manifold $(M,g,K)$.
    \begin{equation}\label{eqProd}
        \partial_i \circ \partial_j=K(\partial_i ,
        \partial_j),
    \end{equation}
where $\partial_i=\partial/\partial x^i$.

\begin{remark}
The definition \eqref{eqProd} is essentially the same as
\cite[Section 4]{JTZ}, where instead of $K$ they used the
Amari-Chentsov tensor, hence differs with \eqref{eqProd} by a factor
$-2$.
\end{remark}

    \begin{proposition}\label{Prop:main}
        On statistical manifold $(M,g,K)$, the product \eqref{eqProd} is associative
        if and only if $M$ has zero sectional $K$-curvature.
        The assertion holds pointwise.
    \end{proposition}
    \begin{proof}
    By the symmetry of $K$, we have
    \begin{equation*}
     \partial_i\circ (\partial_j\circ \partial_k)
     =K(\partial_i,K(\partial_j,\partial_k))=K(\partial_i,K(\partial_k,\partial_j))=K_{\partial_i}K_{\partial_k}\partial_j
    \end{equation*}
    and
    \begin{equation*}
     (\partial_i\circ \partial_j)\circ \partial_k
     =K(K(\partial_i,\partial_j),\partial_k)=K(\partial_k,K(\partial_i,\partial_j))=K_{\partial_k}K_{\partial_i}\partial_j.
    \end{equation*}
     Therefore $\partial_i\circ (\partial_j\circ \partial_k)=(\partial_i\circ \partial_j)\circ \partial_k$ if and only if
     $K_{\partial_i}K_{\partial_k}\partial_j-K_{\partial_k}K_{\partial_i}\partial_j=0$.
     The latter is just
     $$[K,K](\partial_i,\partial_j)\partial_k=0,$$
     which is equivalent to that the sectional $K$-curvatures are
     zero
     by Lemma \ref{Lem:ZeroK}.
\end{proof}

In fact, Proposition \ref{Prop:main} is also a consequence of the
following purely algebraic statement.
\begin{proposition}
If a set $S$ is equipped with a commutative product $S\times
S\rightarrow S$, then the product is associative if and only if left
multiplication operators $l_a$ (where $l_a(x)=ax$) commute.
\end{proposition}
\begin{proof}
If the associativity holds, namely $a(bc)=(ab)c$ for all $a,b,c\in
S$, then
$$a(bc)=(ab)c=(ba)c=b(ac).$$
We proved $l_a$ and $l_b$ commute.

If left multiplication operators commute, namely $a(bc)=b(ac)$ for
all $a,b,c\in S$, then
$$a(bc)=a(cb)=c(ab)=(ab)c,$$
namely the product is associative.
\end{proof}

\begin{proposition}\label{Prop:unit}
On a statistical manifold $(M,g,K)$ with zero sectional
$K$-curvature, the product \eqref{eqProd} has a unit if and only if
$K$ is non-degenerate, i.e., the map $X\rightarrow K_X$ is a
monomorphism.
\end{proposition}
\begin{proof}
By Corollary \ref{Cor:Zero}, $K$ is non-degenerate if and only if
all $\lambda_i$ in \eqref{eqK7} are nonzero. The unit $e$ is given
by
\begin{equation}\label{eqUnit}
e=\sum_{i=1}^{n}\frac{1}{\lambda_i}e_i.
\end{equation}
Corollary \ref{Cor:Zero} also implies that the algebra structure on
$TM$ is semisimple.
\end{proof}

\begin{remark}\label{Rm:Frob}
Propositions \ref{Prop:main} and \ref{Prop:unit} imply that a
statistical manifold $(M,g,K)$ with zero sectional $K$-curvature
(i.e., $[K,K]=0$) and non-degenerate $K$ has a semisimple Frobenius
algebra structure on its tangent space.
\end{remark}

\begin{remark}
If $g$ is a pseudo-Riemannian metric, Proposition \ref{Prop:main}
still holds, but Lemma \ref{Lem:Basis}, Corollary \ref{Cor:Zero} and
Proposition \ref{Prop:unit} may be no longer true, since their
proofs rely on the diagonalizability of $K_X$.
\end{remark}

On a statistical manifold $(M,g,\nabla)$, if we further assume that
the affine connection $\nabla=\hat\nabla+K$ is flat, then it is
called a Hessian manifold. They are so named because on a Hessian
manifold, the Riemannian metric $g$ can be locally expressed as
\begin{equation}\label{eqHess}
g_{ij}=\frac{\partial^2\varphi}{\partial x^i\partial x^j}
\end{equation}
for some function $\varphi$ on affine coordinates induced by the
flat connection $\nabla$. Since the dual connection
$\overline\nabla$ of a Hessian manifold is also flat, they are also
called dually flat manifold. We already see in Example \ref{Ex:exp}
that the exponential family is a Hessian manifold.

\begin{corollary}\label{Cor:hess}
The algebra structure \eqref{eqProd} on a Hessian manifold is
associative if and only if the Levi-Civita connection is flat. In
such case, the $\alpha$-connections are flat for all
$\alpha\in\mathbb R$.
\end{corollary}
\begin{proof}
On a Hessian manifold, we have $R=\overline R=0$. So the first
assertion follows from Proposition \ref{Prop:main} and \eqref{eqR2}.
The last assertion is implied by Lemma \ref{Lem:Alpha}.
\end{proof}

The curvature tensor of the Hessian metric \eqref{eqHess} is given
by (see \cite{Tot})
\begin{equation}
\hat
R_{ijkl}=\frac{1}{4}g^{pq}(\varphi_{ilp}\varphi_{jkq}-\varphi_{ikp}\varphi_{jlq}),
\end{equation}
thus the vanishing condition of $\hat R$ is the famous WDVV
equation:
\begin{equation}\label{eqWDVV}
g^{pq}(\varphi_{ilp}\varphi_{jkq}-\varphi_{ikp}\varphi_{jlq})=0.
\end{equation}
Since WDVV equation is equivalent to the associativity, we got
another proof of Corollary \ref{Cor:hess}. Although originated from
string theory, WDVV equation appears in many problems of
differential geometry. See \cite[Section 2]{Tot} for an excellent
summary.

It is well-known that canonical basis exists around a semisimple
point of a Frobenius manifold. The set of all semisimple points on a
Frobenius manifold is open.
\begin{proposition}\label{Prop:Basis}\cite{Hit}
Around any point on a manifold with semisimple Frobenius structure,
there exists a canonical basis $u_1,\dots,u_n$, which are local
vector fields that satisfy
$$u_i^2=u_i,\qquad u_i u_j=0$$
for $i\neq j$. They are unique up to reordering.
\end{proposition}

A different proof of Proposition \ref{Prop:Basis} was given in
\cite{Opo} under the setting of statistical manifold where a
technical smoothing argument was used.

Denote $C(X,Y,Z)=g(XY,Z)$ on a manifold with Frobenius structure.
\begin{proposition}\label{Prop:Basis2}\cite{Hit}
In Proposition \ref{Prop:Basis}, $\hat\nabla C$ is symmetric if and
only if
\begin{align}
g(\hat\nabla_{u_k}u_i,u_j)&=0 \label{equ}\\
g(\hat\nabla_{u_i}u_i,u_j)&=g(\hat\nabla_{u_j}u_j,u_i) \label{equ2}\\
g(\hat\nabla_{u_i}u_j,u_j)&=g(\hat\nabla_{u_i}u_j,u_i)=g(\hat\nabla_{u_j}u_i,u_j)
\label{equ3}
\end{align}
for distinct $i,j,k$.

In fact, when $\hat\nabla C$ is symmetric, there exists a local
orthogonal coordinate system $x^1,\dots,x^n$ such that
$u_i=\dfrac{\partial}{\partial x^i}$ and there is a function $\phi$
such that
$$g=\sum_i\frac{\partial\phi}{\partial x^i} dx^i\otimes dx^i.$$
Such a metric $g$ is called Darboux-Egoroff metric.
\end{proposition}

\begin{corollary}\label{Cor:C}
Let $(M,g,K)$ be a statistical manifold with non-degenerate $K$ and
$[K,K]=0$. Then $\hat\nabla C=0$ if and only if $\hat\nabla u_i=0$
for all $i$. In particular, $\hat\nabla C=0$ implies that $g$ is
flat.
\end{corollary}
\begin{proof}
By Remark \ref{Rm:Frob}, the results of Proposition
\ref{Prop:Basis2} may apply. We follow Hitchin's argument in the
proof of Proposition \ref{Prop:Basis2}. Note that $C(u_i,u_j,u_k)=0$
unless $i=j=k$.
\begin{equation*}
(\hat\nabla_{u_i}C)(u_j,u_k,u_l)=u_i
C(u_j,u_k,u_l)-C(\hat\nabla_{u_i}u_j,u_k,u_l)-C(u_j,\hat\nabla_{u_i}u_k,u_l)-C(u_j,u_k,\hat\nabla_{u_i}u_l).
\end{equation*}
Assume $\hat\nabla C=0$. Let $i\neq j$.

Setting $k=l=i$ gives
$g(\hat\nabla_{u_i}u_j,u_i)=C(\hat\nabla_{u_i}u_j,u_i,u_i)=0$.

Setting $k=l=j$ and using $u_i
g(u_j,u_j)=2g(\hat\nabla_{u_i}u_j,u_j)$ gives
$g(\hat\nabla_{u_i}u_j,u_j)=0$.

By \eqref{equ}, $g(\hat\nabla_{u_i}u_j,u_k)=0$ for distinct $i,j,k$.
So we proved
\begin{equation}\label{equ4}
\hat\nabla_{u_i}u_j=0,\quad i\neq j.
\end{equation}

Setting $i=j=k=l$ gives $u_i
g(u_i,u_i)=3g(\hat\nabla_{u_i}u_i,u_i)$. Hence
$g(\hat\nabla_{u_i}u_i,u_i)=0$.

Setting $i=j$ and $k=l$ gives $g(\hat\nabla_{u_i}u_i,u_j)=0$ for
$i\neq j$. So we proved
\begin{equation}\label{equ5}
\hat\nabla_{u_i}u_i=0,\quad \forall i.
\end{equation}
By \eqref{equ4} and \eqref{equ5}, we proved that $\hat\nabla C=0$
implies $\hat\nabla u_i=0$.

On the other hand, if $\hat\nabla u_i=0$, then
$(\hat\nabla_{u_i}C)(u_j,u_k,u_l)=0$ unless $j=k=l$. Note that
$$(\hat\nabla_{u_i}C)(u_j,u_j,u_j)=u_i C(u_j,u_j,u_j)=u_i g(u_j,u_j)=2g(\hat\nabla_{u_i}u_j,u_j)=0.$$
So we proved $\hat\nabla C=0$.
\end{proof}

The trace of $K$ is defined to be the vector field $E=\mathrm{tr}_g
K$, namely
\begin{equation}
E^h=g^{ij}K_{ij}^h.
\end{equation}

\begin{corollary}\label{Cor:C2}
Let $(M,g,K)$ be a statistical manifold with non-degenerate $K$ and
$[K,K]=0$. Then $\hat\nabla C$ is symmetric and $\hat\nabla E=0$ if
and only if $\hat\nabla C=0$ and $g(u_i,u_i)$ are constants for all
$i$.
\end{corollary}
\begin{proof}
By Corollary \ref{Cor:Zero}, there is an orthonormal basis
$e_1,...,e_n$ such that
\begin{equation*}
K(e_i,e_i)=\lambda_i e_i,\qquad K(e_i,e_j)=0
\end{equation*}
for $i\neq j$. Then
\begin{equation}
u_i=\frac{1}{\lambda_i}e_i,\qquad g(u_i,u_i)=\frac{1}{\lambda_i^2},
\end{equation}
where $\lambda_i$ are smooth functions. Let
\begin{equation}
f_{ij}=g(\hat\nabla_{u_i} u_j,u_j).
\end{equation}
Assume $\hat\nabla C$ is symmetric. Then by \eqref{equ},
\eqref{equ2} and \eqref{equ3}, we get $f_{ij}=f_{ji}$ and
\begin{align}
\nabla_{u_i} u_i&=\lambda_i^2 f_{ii}u_i-\sum_{j\neq
i}\lambda_j^2f_{ij}u_j \label{equ6} \\
\nabla_{u_i} u_j&=\lambda_i^2 f_{ij}u_i+\lambda_j^2
f_{ij}u_j.\label{equ7}
\end{align}
for $i\neq j$. From $u_i\cdot g(u_j,u_j)=2g(\hat\nabla_{u_i}
u_j,u_j)=2f_{ij}$, we get
\begin{equation}\label{equ8}
u_i(\lambda_j^2)=-2\lambda_j^4f_{ij}.
\end{equation}
Since $E=\sum_{j=1}^n\lambda_j^2 u_j$, then by \eqref{equ6},
\eqref{equ7} and \eqref{equ8}, we get
\begin{align*}
\hat\nabla_{u_i}E&=\sum_{j=1}^n u_i(\lambda_j^2)u_j+\sum_{j=1}^n \lambda_j^2\hat\nabla_{u_i}u_j\\
&=-2\sum_{j=1}^n \lambda_j^4f_{ij}u_j+\lambda_i^4
f_{ii}u_i-\lambda_i^2\sum_{j\neq i}\lambda_j^2f_{ij}u_j +\sum_{j\neq
i}\lambda_j^2(\lambda_i^2 f_{ij}u_i+\lambda_j^2 f_{ij}u_j)\\
&=\left(-\lambda_i^4f_{ii}+\sum_{j\neq
i}\lambda_i^2\lambda_j^2f_{ij}\right)u_i-\sum_{j\neq
i}\lambda_j^2(\lambda_i^2+\lambda_j^2)f_{ij}u_j.
\end{align*}
So $\hat\nabla_{u_i}E=0$ implies $f_{ij}=0$ for all $i,j$, hence
$\hat\nabla_{u_i} u_j=0$ for all $i,j$. By \eqref{equ8}, all
$\lambda_i$ are constants.

The reverse direction of the corollary is obvious.
\end{proof}

Corollaries \ref{Cor:C} and \ref{Cor:C2} provide alternative proofs
to Theorem 4.6 and part of Corollary 4.7 in \cite{Opo}. The above
results showed that the behavior of $\hat\nabla C$ or $\hat\nabla K$
gave strong constraints to the metric $g$. See \cite{Opo} for more
related results.

\begin{remark}
Semisimplicity is a very important property for Frobenius manifolds.
They correspond to integrable hierarchies of KdV type \cite{DZ}.
Opozda's algorithm that we described in Remark \ref{Rm:Algo} could
be used to explicitly calculate canonical basis at a semisimple
point of Frobenius manifold.
\end{remark}

In \cite{JTZ}, an invariant of statistical manifold called Yukawa
term was introduced. It is defined by
\begin{equation}\label{eqY2}
Y=C_{ijk}C^{ijk}-C_i C^i,
\end{equation}
where $C_i=C_{ijk}g^{jk}$. In fact, they used the Amari-Chentsov
tensor, hence their Yukawa term differs with \eqref{eqY2} by a
factor $4$.
\begin{proposition}\label{Prop:Yuk}
On a statistical manifold, if $[K,K]=0$, then the Yukawa term $Y=0$.
In $2$-dimension, the converse is also true.
\end{proposition}
\begin{proof}
By Corollary \ref{Cor:Zero}, in the orthonormal basis
$e_1,\dots,e_n$,
$$C_{iii}=C^{iii}=\lambda_i,\qquad C_i=C^i=\lambda_i.$$
All other $C_{ijk}$ vanish. So $Y=n\lambda^2-n\lambda^2=0$.

For the last assertion in $2$-dimension, see \eqref{eqK4} and
\eqref{eqY} in Example \ref{Ex:2d}.
\end{proof}

From Proposition \ref{Prop:Yuk}, the vanishing of the Yukawa term is
a necessary condition for zero sectional $K$-curvature.

\begin{example}\label{Ex:2d}
Consider the isothermal coordinates on a $2$-manifold with metric
$$g(x,y)=\varphi(x,y)(dx\otimes dx+dy\otimes dy).$$
The symmetric tensor
$$C_{ijk}=K_{ij}^l g_{kl}=\varphi K_{ij}^k$$
has four independent components
$$f_1=C_{111},\quad f_2=C_{112},\quad f_3=C_{122},\quad f_4=C_{222}.$$
Then $[K,K]=0$ if and only if
\begin{equation}\label{eqK4}
f_2^2+f_3^2=f_1 f_3+f_2 f_4.
\end{equation}
The solutions consist of three families:
\begin{itemize}
\item[(i)] $f_3=0,f_2=f_4$,
\item[(ii)] $f_2=f_3=0$,
\item[(iii)] $f_3\neq 0,f_1=\frac{1}{f_3}(f_2^2+f_3^2-f_2 f_4)$.
\end{itemize}
$K$ is non-degenerate if and only if
\begin{equation}\label{eqK5}
\mathrm{rank}\begin{pmatrix} f_1 & f_2 & f_3 \\ f_2 & f_3 & f_4
\end{pmatrix}=2.
\end{equation}
When both \eqref{eqK4} and \eqref{eqK5} are satisfied, the unit $e$
in the Frobenius algebra is
\begin{itemize}
\item[(i)] $e=\frac{\varphi}{f_2}\partial_y$,

\item[(ii)]
$e=\frac{\varphi}{f_1}\partial_x+\frac{\varphi}{f_4}\partial_y$,

\item[(iii)] $e=\frac{\varphi f_3}{f_1 f_3-f_2^2}\partial_x-\frac{\varphi f_2}{f_1
f_3-f_2^2}\partial_y$,
\end{itemize}
corresponding to each of the three families of solutions of
\eqref{eqK4}.

The Yukawa term is given by
\begin{equation}\label{eqY}
Y=\frac{2}{\varphi^3}(f_2^2+f_3^2-f_1 f_3-f_2 f_4).
\end{equation}
\end{example}

Let $M$ be a statistical manifold with Frobenius structure. If
$\hat\nabla$ is flat and $\hat\nabla C$ is symmetric, then by the
Poincar\'e Lemma, there exist a potential function $F$, such that
\cite[Section 3.3]{Hit} $C_{ijk}=F_{ijk}$, where
\begin{equation}
F_{ijk}=\frac{\partial^3 F}{\partial x_i\partial x_j\partial x_k}
\end{equation}
and $x_1,\dots,x_n$ is the flat coordinates corresponding to
$\hat\nabla$.

Let $B$ be a flat pseudo-Riemannian metric on $M$. Denote by
$$e=\sum_{i=1}^n A_i\partial_i$$
the unit vector field, where $A_i$ are smooth functions. Then
\begin{equation}\label{eqB}
B_{ij}=C(\partial_i,\partial_j,e)=\sum_{k=1}^n A_k F_{ijk}.
\end{equation}
Let $F_i$ be the $n\times n$ matrix $(F_i)_{jk}=F_{ijk}$. The WDVV
equations are the following equations.
\begin{equation}\label{eqWDVV2}
F_{i}B^{-1}F_{j}=F_{j}B^{-1}F_{i}
\end{equation}
for all $i,j$.

There is large amount of work solving the potential WDVV equation
\eqref{eqWDVV2}, we only mention a solution related to $BC_n$ root
system in a recent paper of Alkadhem-Antoniou-Feigin \cite{AAF}.
\begin{theorem}\label{Thm:AAF}\cite{AAF}
Suppose numbers $r$, $s$ and $q$ satisfy $r=-8s-2q(n-2)$. Let
$$f(z)=\frac{1}{6}z^3-\frac{1}{4}\mathrm{Li}_3(e^{-2z}),$$
where $\mathrm{Li}_3$ is the trilogarithm function. Then the
function
\begin{align}\label{prepBCN}
F=r \sum_{i=1}^{n}f(x_{i})+s\sum_{i=1}^{n}f(2x_{i})+q\sum_{1\leq
i<j\leq n}\big(f(x_{i}+x_{j})+f(x_{i}-x_{j})\big)
\end{align}
satisfies WDVV equations \eqref{eqWDVV2} where $B$ is determined by
\eqref{eqB} and $A_k=\sinh(2x_k)$. In fact, $B$ is proportional to
the identity matrix
$$B_{ij}=h\delta_{ij},$$
where $h(x)=r+2q\sum_{k=1}^n\cosh 2x_k$.

\end{theorem}

Finally, we gave two remarks. It would be interesting to study
statistical structures with constant sectional $K$ curvatures. As
noted in \cite{Opo}, in general we don't know whether the basis in
Lemma \ref{eqWDVV2} could be extended to be a smooth frame field,
although it works for trace-free statistical structure.

We know very few results on the global aspect of statistical
manifold with Frobenius structure. Hitchin \cite{Hit} showed that
there is a finite covering of such manifold which is parallelizable.
A first example to look is the compact Hessian manifolds. There is
the famous Chern conjecture \cite{FZ} that a flat manifold has zero
Euler characteristic, which is still open even for Hessian manifold.

$$ \ \ \ \ $$


\begin{thebibliography}{999}

\bibitem{AAF} M.Alkadhem, G. Antoniou and M. Feigin, {\em Solutions of $BC_n$ type of WDVV
equations}, arXiv:2002.02900.

\bibitem{Ama} S. Amari, {\em Information geometry and its applications}, Springer, 2016.

\bibitem{BK} S. Barannikov and M. Kontsevich, {\em Frobenius manifolds and
formality of Lie algebras of polyvector fields}, Int. Math. Res.
Notices 1998, 201--215.

\bibitem{CY} S.-Y. Cheng and S.-T. Yau, {\em The real Monge-Amp\`ere equation and
affine flat structures}, in Proceedings of the 1980 Beijing
Symposium on Differential Geometry and Differential Equations, Vol.
1, Science Press, Beijing, 1982, 339--370.

\bibitem{CCN} N. Combe, P. Combe and H. Nencka, {\em Frobenius statistical manifolds {\rm \&} geometric
invariants}, arXiv:2107.08446.

\bibitem{CM} N. Combe and Yu. I. Manin, {\em $F$-manifolds and geometry of
information}, Bull. London Math. Soc. {\bf52} (2020) 777--792.

\bibitem{Dub} B. Dubrovin, {\em Painlev\'e transcendents in two-dimensional topological field
theory}, In: ``The Painlev\'e Property One Century Later'',
287--412, Springer, 1999.

\bibitem{DZ} B. Dubrovin and Y. Zhang, {\em Frobenius manifolds and Virasoro
constraints}, Sel. Math., New Ser. {\bf5}, 423--466 (1999).

\bibitem{FZ} H. Feng and W. Zhang, {\em Flat vector bundles and open
coverings}, arXiv:1603.07248.

\bibitem{Hit} N. Hitchin, {\em Frobenius manifolds}, In Gauge Theory and
Symplectic Geometry, 69--112. Kluwer Academic Publishers, 1997.

\bibitem{JTZ} R. Jiang, J. Tavakoli and Y. Zhao, {\em Information geometry and Frobenius algebra}, arXiv:2010.05110 (2020).

\bibitem{Man} Yu. I. Manin, {\em Frobenius manifolds, quantum cohomology, and moduli
spaces}, volume 47 of AMS Colloquium Publications. AMS, Providence,
RI, 1999.

\bibitem{NO} N. Nakajima, T. Ohmoto, {\em The dually flat structure for singular
models}, arXiv:2011.11219.

\bibitem{NS} K. Nomizu and T. Sasaki, {\em Affine differential geometry}, Cambridge University Press, 1994.

\bibitem{Opo} B. Opozda, {\em A sectional curvature for statistical structures}, Linear Algebra Appl. {\bf 497} (2016) 134--161.

\bibitem{Sai} K. Saito, {\em Period mapping associated to a primitive form}, Publ.
Res. Inst. Math. Sci. {\bf19} (1983), 1231--1264.

\bibitem{Tot} B. Totaro, {\em The curvature of a Hessian metric}, Internat. J. Math.
{\bf15} (2004), 369--391.

\bibitem{Zha} J. Zhang, {\em A note on curvature of $\alpha$-connections}, AISM, {\bf59} (2007), 161--170.








\end{thebibliography}
\end{document}